\documentclass[11pt]{amsart}

\usepackage[T1]{fontenc}
\usepackage{lmodern}
\usepackage{microtype}
\usepackage{amsmath,amssymb,amsthm,mathtools,mathrsfs}
\usepackage{enumitem}
\usepackage[colorlinks=true,linkcolor=blue,citecolor=blue,urlcolor=blue]{hyperref}
\hypersetup{
  pdftitle={Right q-Radial Calculus at Integral Superdimension},
  pdfauthor={Juan Bory-Reyes, Baruch Schneider, Diana Schneiderova, Yifan Zhang},
  pdfkeywords={radial algebra, superspace, q-calculus, Fischer decomposition, resonance}
}
\allowdisplaybreaks
\newtheorem{theorem}{Theorem}[section]
\newtheorem{proposition}[theorem]{Proposition}
\newtheorem{lemma}[theorem]{Lemma}
\newtheorem{corollary}[theorem]{Corollary}
\newtheorem{definition}[theorem]{Definition}
\theoremstyle{remark}
\newtheorem{remark}[theorem]{Remark}
\newtheorem{example}[theorem]{Example}

\newcommand{\R}{\mathbb R}
\newcommand{\C}{\mathbb C}
\newcommand{\N}{\mathbb N}
\newcommand{\B}{\mathbb B}
\newcommand{\K}{\mathbb K}
\newcommand{\cR}{\mathcal R}
\newcommand{\cA}{\mathcal A}

\newcommand{\cH}{\mathcal H}
\newcommand{\cG}{\mathcal G}
\newcommand{\cC}{\mathcal C}
\newcommand{\cK}{\mathcal K}
\newcommand{\cY}{\mathcal Y}
\newcommand{\cX}{\mathcal X}

\newcommand{\Ker}{\operatorname{Ker}}
\newcommand{\im}{\operatorname{im}}

\newcommand{\qnum}[1]{\left[#1\right]_q}
\newcommand{\qnumat}[2]{\left[#1\right]_{#2}}
\newcommand{\partialR}{\partial^{\mathrm R}_{x,q}}
\newcommand{\partialRY}{\partial^{Y,\mathrm R}_{x,q}}
\newcommand{\partialsc}{\partial^{\mathrm{sc},Y}_{x,q}}
\newcommand{\RB}{R_{\B}}

\title[Right $q$-radial calculus]{Right $q$-Radial Calculus at Integral Superdimension}

\author[J.~Bory-Reyes]{Juan Bory-Reyes}
\address{ESIME-Zacatenco, Instituto Polit\'ecnico Nacional, 07738 CDMX, M\'exico}
\email{juanboryreyes@yahoo.com}

\author[B.~Schneider]{Baruch Schneider}
\address{Department of Mathematics, University of Ostrava, 70200 Ostrava, Czechia}
\email{baruch.schneider@osu.cz}

\author[D.~Schneiderov\'a]{Diana Schneiderov\'a}
\address{Department of Mathematics, University of Ostrava, 70200 Ostrava, Czechia}
\email{diana.schneiderova@osu.cz}

\author[Y.~Zhang]{Yifan Zhang}
\address{Department of Mathematics, University of Ostrava, 70200 Ostrava, Czechia}
\address{Department of Applied Mathematics, VSB--Technical University of Ostrava, 70800 Ostrava, Czechia}
\email{yifan.zhang@osu.cz}

\date{}

\begin{document}

\begin{abstract}
We study the intrinsic right $q$-radial derivative at integral superdimension
$M=m-2n$, using an independent formal parameter $Q$ before the specialisation
$Q=q^M$. The exterior-blade action carries the parity factor $(-1)^{|I|}$,
and the corresponding Fischer theory uses right multiplication. We establish
finite-support compatibility, faithful coordinate realisations in a sufficient
finite-block range, and localised Green and determinant-localised Fischer decompositions.
For one vector, we determine the kernel and cokernel explicitly, including the
additional singular monomial and missing image exactly when
$M\in-2\mathbb N_0$. In degree zero the right Fischer operator is diagonal
on auxiliary exterior blades, with eigenvalue $[M-p]_q+p$ at nonempty
support rank $p$. For positive integral $M$, resonances in $0<q<1$ occur
exactly at ranks $p>M$: each rank has one simple resonance, these values
increase strictly with $p$, and the corresponding kernels and singular
right-monogenic elements are explicit. The two-vector Fischer determinant
is nonzero for positive $M$ and $0<q<1$.
\end{abstract}

\maketitle

\section{Introduction}

The radial algebra of abstract Clifford vectors captures invariant
relations in Clifford analysis while leaving the dimension as a parameter
\cite{Sommen1997,DeSchepperGuzmanSommen2017}. The intrinsic right
$q$-radial derivative of \cite{BarseghyanBorySchneiderZhang2026} is
constructed using whole-vector dilation and scalar Jackson differences.
Its right-action convention is fixed by the Clifford relations
$e_ie_j+e_je_i=-2\delta_{ij}$ and the identities
\[
 [F]\partial_x=-\sum_j(\partial_{x_j}F)e_j,
 \qquad [(x\wedge y_I)]\partial_x=(-1)^{|I|}(M-|I|)y_I.
\]
The associated right Fischer operator pairs differentiation with right
multiplication $R_xF=Fx$. The parity and multiplication-side conventions
are essential in determining its singular values.

We consider the integral-superdimension specialisation $Q\mapsto q^M$,
where $M=m-2n$ may be any integer. Formal base change and faithful
coordinate representations in sufficiently large bosonic dimension give
complementary interpretations. We examine finite Green and Fischer
operators, determine their exceptional parameters, and describe the
resulting kernels. In the one-vector case, the kernel and image follow
from an explicit monomial lowering rule. For arbitrary finite auxiliary
support, an alternating Clifford identity gives an exact degree-zero
blade spectrum, whose eigenvalues are
\[
 F_{M,0}(q)=[M]_q,\qquad
 F_{M,p}(q)=[M-p]_q+p\quad(p\ge1).
\]
We classify the numerical roots of these factors and their multiplicities.
The two-vector determinant supplies an all-degree benchmark, using the
right Fischer pairing of \cite{BarseghyanBorySchneiderZhang2026}.

Section~\ref{sec:radial-q} defines the right calculus and examines its
behaviour under quadratic constraints. Sections~\ref{sec:superspace}
and~\ref{sec:finite-super} establish the formal and finite-coordinate
superspace realisations. Section~\ref{sec:decompositions} develops localised
decompositions, and Section~\ref{sec:resonance} studies their singular
behaviour.

\section{Radial algebras and the right \texorpdfstring{$q$}{q}-vector derivative}\label{sec:radial-q}

Let $S$ be a set of abstract vector variables.  The radial algebra $R(S)$ is the associative $\R$-algebra generated by $S$ subject to
\begin{equation}\label{eq:radial-relation}
        [\{u,v\},w]=0,
        \qquad u,v,w\in S.
\end{equation}
Its scalar subalgebra is generated by the central anticommutators.  Every element of a finitely generated radial algebra has a unique exterior expansion with scalar coefficients \cite{Sommen1997,DeSchepperGuzmanSommen2017}.

\subsection{Coefficient ring and specialisation}

Let $q$ and $Q$ be algebraically independent and put
\begin{equation}\label{eq:coefficient-ring}
\begin{aligned}
 \mathbb A_q&:=\R[q,q^{-1},(1-q)^{-1}],
 &\qquad \B&:=\mathbb A_q[Q],\\
 \K&:=\operatorname{Frac}(\B)=\R(q,Q),
 &\qquad \RB(S)&:=\B\otimes_{\R}R(S).
\end{aligned}
\end{equation}
We denote the formal dimension by $d$.  For $a\in\mathbb Z$ set
\begin{equation}\label{eq:formal-dimension}
        q^{d+a}:=Qq^a,
        \qquad
        \qnum{d+a}:=\frac{1-Qq^a}{1-q},
\end{equation}
and for $r\in\mathbb Z$ write
\begin{equation}\label{eq:integer-q-number}
        \qnum r:=\frac{1-q^r}{1-q}.
\end{equation}
Thus $q^d$ is mnemonic notation for the independent coefficient $Q$.

For every integer $M$, define the specialisation
\begin{equation}\label{eq:sigma-M}
        \sigma_M:\B\longrightarrow\mathbb A_q,
        \qquad
        \sigma_M(Q)=q^M.
\end{equation}
The use of $q^{-1}$ in $\mathbb A_q$ makes \eqref{eq:sigma-M} well defined also for negative $M$.  After specialisation, $\qnum{d+a}$ becomes $\qnum{M+a}$.  Classical limits are taken after applying $\sigma_M$ and then letting $q\to1$.

For a scalar variable $u$ and a parameter $P$, the Jackson difference and dilation are
\begin{equation}\label{eq:jackson}
        \delta_P^{(u)}f
        :=\frac{f(u)-f(Pu)}{(1-P)u},
        \qquad
        T_P^{(u)}f(u):=f(Pu).
\end{equation}
All localisations below are by multiplicative subsets of central coefficient rings.

\subsection{Finite relative and direct-limit derivatives}

Fix $x\in S$ and a finite auxiliary set
\[
        Y=\{y_1,\ldots,y_N\}\subset S\setminus\{x\}.
\]
Put
\begin{equation}\label{eq:relative-scalars}
        r=x^2,
        \qquad s_i=\{x,y_i\},
        \qquad c_{ij}=\{y_i,y_j\},
\end{equation}
and let
\[
        \cA_x(Y)=\B[r,s_1,\ldots,s_N,c_{ij}:1\le i\le j\le N].
\]
For $I=\{i_1<\cdots<i_p\}$ write $y_I=y_{i_1}\wedge\cdots\wedge y_{i_p}$.  The $x$-relative exterior decomposition is
\begin{equation}\label{eq:relative-decomposition}
\RB(\{x\}\cup Y)
 =\bigoplus_I y_I\cA_x(Y)
 \oplus\bigoplus_I(x\wedge y_I)\cA_x(Y).
\end{equation}

The scalar-input operator is
\begin{equation}\label{eq:scalar-operator}
        \partialsc
        =(1+q)x\delta_{q^2}^{(r)}
          +2T_{q^2}^{(r)}\sum_{i=1}^Ny_i\delta_q^{(s_i)}.
\end{equation}
The finite relative right $q$-vector derivative of \cite{BarseghyanBorySchneiderZhang2026} is the $\B$-linear map determined by
\begin{align}
\partialRY(y_IA)&=y_I\partialsc(A),\label{eq:finite-y}\\
\partialRY((x\wedge y_I)B)
&=(-1)^{|I|}\qnum{d-|I|}y_IB+q^{d-|I|}(x\wedge y_I)\partialsc(B),\label{eq:finite-x}
\end{align}
where scalar coefficients are written on the right. The parity factor agrees with the classical right Clifford action; in particular, $[(x\wedge y)]\partial_x=-(M-1)y$. For $Y=\varnothing$ these formulas give
\begin{equation}\label{eq:one-vector-formal}
        \partialR(x^{2a})=\qnum{2a}x^{2a-1},
        \qquad
        \partialR(x^{2a+1})=\qnum{d+2a}x^{2a}.
\end{equation}
The limit after $Q=q^M$, $q\to1$ is the classical right radial vector derivative.

\begin{theorem}[Direct-limit compatibility]\label{thm:direct-limit}
If $Y\subset Z\subset S\setminus\{x\}$ are finite, then
\[
 \partial^{Z,\mathrm R}_{x,q}\bigm|_{\RB(\{x\}\cup Y)}
 =\partial^{Y,\mathrm R}_{x,q}.
\]
Consequently the finite operators define a $\B$-linear operator on the direct limit
\[
        \partialR:\RB(S)\longrightarrow\RB(S).
\]
\end{theorem}

\begin{proof}
Every element involves finitely many vector variables.  If $Y$ is enlarged to $Z$, its exterior expansion contains no blade involving the new variables, and every Jackson difference with respect to a newly introduced mixed scalar vanishes.  Equations \eqref{eq:finite-y}--\eqref{eq:finite-x} therefore give the same value.  Covariance under relabelling, proved in \cite{BarseghyanBorySchneiderZhang2026}, removes the dependence on the ordering of the finite auxiliary set.
\end{proof}

\subsection{Quadratic Hermitian constraints}\label{subsec:quadratic-constraints}

The preservation of quadratic relations can be tested within the
finite-support calculus. Let $x,u$ be independent abstract vector labels,
extend scalars to $\mathbb F=\C(q,Q)$, and let $I_J$ be the two-sided ideal
in $\mathbb F\otimes_{\R}R(\{x,u\})$ generated by
\begin{equation}\label{eq:hermitian-constraints}
  h:=u^2-x^2,\qquad s:=\{x,u\}.
\end{equation}
These are the quadratic relations of a Hermitian pair $u=J(x)$.

\begin{proposition}[Constant-projection obstruction]\label{prop:hermitian-constraint-obstruction}
Let $D_{x,q}^{\mathrm R}$ and $D_{u,q}^{\mathrm R}$ denote the independent
right $q$-radial derivatives and put
$L=aD_{x,q}^{\mathrm R}+bD_{u,q}^{\mathrm R}$, where $a,b\in\mathbb F$.
If $L(I_J)\subseteq I_J$, then $a=b=0$.
\end{proposition}
\begin{proof}
The scalar-input rules give
\[
 D_{x,q}^{\mathrm R}(x^2)=[2]_q x,\quad
 D_{x,q}^{\mathrm R}(u^2)=0,\quad
 D_{x,q}^{\mathrm R}\{x,u\}=2u,
\]
with the corresponding identities after interchanging $x$ and $u$. Hence
\[
 L(h)=[2]_q(-ax+bu),\qquad L(s)=2(bx+au).
\]
The ideal $I_J$ is homogeneous and generated in degree two, so it contains
no nonzero elements of degree one. The degree-one classes of $x$ and $u$
are linearly independent and $[2]_q$ is nonzero in $\mathbb F$.
The inclusion $L(I_J)\subseteq I_J$ therefore forces $a=b=0$.
\end{proof}

\section{Formal and coordinate radial superspace}\label{sec:superspace}

Let $m,n\in\N_0$.  The standard Clifford algebra of superspace is generated by commuting bosonic coordinates, anticommuting fermionic coordinates, orthogonal Clifford generators, and symplectic Clifford generators; see \cite{DeBieSommen2007AnnPhys,DeBieSommen2007CorrectRules}.  A supervector has the form
\begin{equation}\label{eq:supervector}
        x=\sum_{j=1}^{m}x_je_j+
          \sum_{k=1}^{2n}\grave{x}_k\grave e_k.
\end{equation}
For the classical right super Dirac operator one has
\begin{equation}\label{eq:superdimension-evaluation}
        [x]\partial_x^{\mathrm R}=m-2n=:M.
\end{equation}

\subsection{Formal radial superspace by base change}

\begin{definition}[Formal radial superspace algebra]\label{def:formal-radial-super}
For $M\in\mathbb Z$, define first the regular coefficient form
\begin{equation}\label{eq:regular-formal-super-algebra}
 \cR_M^{\mathbb A}(S)
 :=\mathbb A_q\otimes_{\sigma_M,\B}\RB(S)
 \cong\mathbb A_q\otimes_{\R}R(S),
\end{equation}
and then its fraction-field extension
\begin{equation}\label{eq:formal-super-algebra}
 \cR_M(S):=\R(q)\otimes_{\mathbb A_q}\cR_M^{\mathbb A}(S).
\end{equation}
The specialised right $q$-vector derivative is defined on the regular form by
\begin{equation}\label{eq:specialised-derivative}
 1\otimes\partialR:\cR_M^{\mathbb A}(S)
 \longrightarrow\cR_M^{\mathbb A}(S),
\end{equation}
and is denoted by $\partial^{M,\mathrm R}_{x,q}$ also after extension to $\cR_M(S)$.  The same convention applies to every finite relative operator.
\end{definition}

This universal definition records the radial identities that depend on superspace only through its superdimension; coordinate syzygies do not enter this construction.  The regular form in \eqref{eq:regular-formal-super-algebra} is retained because numerical evaluation of $q$ is not defined on all of $\R(q)$.

\begin{theorem}[Base-change superdimension principle]\label{thm:base-change}
Every $\B$-linear operator identity built from the finite or direct-limit right $q$-vector derivative, vector multiplication, exterior creation, and central scalar coefficients remains valid on $\cR_M(S)$ after applying $\sigma_M$.

Let $\Sigma$ be a central multiplicative set in a coefficient ring occurring in a finite block.  If $0\notin\sigma_M(\Sigma)$, then specialisation commutes with localisation and with every inverse defined over that localisation.
\end{theorem}

\begin{proof}
Tensor first with $\mathbb A_q$ via $\sigma_M$ and then extend scalars to $\R(q)$.  Both operations preserve sums, compositions, and equality of linear maps.  If $0\notin\sigma_M(\Sigma)$, the universal property of commutative localisation gives
\[
 \R(q)\otimes_{\sigma_M,\B}\RB(S)_\Sigma
 \cong \cR_M(S)_{\sigma_M(\Sigma)}.
\]
Matrices and their inverses specialise entrywise because their determinant factors remain invertible.
\end{proof}

\begin{corollary}[One-vector radial superspace calculus]\label{cor:one-vector-super}
On $\cR_M(\{x\})$,
\begin{equation}\label{eq:one-vector-super}
        \partial^{M,\mathrm R}_{x,q}(x^{2a})=\qnum{2a}x^{2a-1},
        \qquad
        \partial^{M,\mathrm R}_{x,q}(x^{2a+1})=\qnum{M+2a}x^{2a}.
\end{equation}
After $q\to1$, these are the classical radial formulas in superdimension $M$.
\end{corollary}

\subsection{Faithful finite coordinate blocks}

Let
\[
        \pi_{m|2n,T}:R(T)\longrightarrow\cA_{m,2n}
\]
be the standard radial superspace representation of a finite set $T$ of abstract vectors by independent supervector variables.  Extend scalars first from $\R$ to $\mathbb A_q$ and then to $\R(q)$, obtaining representations of both $\cR_M^{\mathbb A}(T)$ and $\cR_M(T)$, where $M=m-2n$.

\begin{proposition}[A sufficient faithfulness range]\label{prop:faithful-super-block}
If $|T|\le m$, then $\pi_{m|2n,T}$ is injective.  Hence the specialised operator \eqref{eq:specialised-derivative} has an unambiguous coordinate realisation on the represented finite radial superspace block.
\end{proposition}

\begin{proof}
Set all fermionic coordinates equal to zero.  Composing $\pi_{m|2n,T}$ with this specialisation gives the ordinary Clifford-polynomial representation of $R(T)$ in bosonic dimension $m$.  That representation is injective when $m\ge |T|$ \cite{Sommen1997}.  If an element lies in the kernel of the superspace representation, then its bosonic restriction is zero; injectivity of the latter representation forces the original radial element to be zero.  Moreover, $\cR_M(T)$ identifies canonically with $\R(q)\otimes_{\R}R(T)$.  Extending the injective $\R$-linear representation from $\R$ to the field $\R(q)$ therefore preserves injectivity.
\end{proof}

\begin{definition}[Coordinate right $q$-super derivative]\label{def:coordinate-super-derivative}
Assume $T=\{x\}\cup Y$ and $|T|\le m$.  On the image
\[
        R_{m|2n}(T;q):=\pi_{m|2n,T}(\cR_M(T))
\]
define
\begin{equation}\label{eq:coordinate-conjugation}
\partial^{m|2n,Y,\mathrm R}_{x,q}
 :=\pi_{m|2n,T}\,\partial^{M,Y,\mathrm R}_{x,q}\,\pi_{m|2n,T}^{-1}.
\end{equation}
\end{definition}

\begin{remark}\label{rem:nonfaithful-descent}
For a nonfaithful coordinate block, \eqref{eq:coordinate-conjugation} is defined precisely when the specialised derivative preserves the representation kernel. The universal identities hold in $\cR_M(S)$; their coordinate counterparts hold in the faithful range of Proposition~\ref{prop:faithful-super-block}, or when kernel invariance is established separately.
\end{remark}

\section{Finite right \texorpdfstring{$q$}{q}-vector calculus in radial superspace}\label{sec:finite-super}

Fix $Y=\{y_1,\ldots,y_N\}$ and $M\in\mathbb Z$.  The scalar variables remain those of \eqref{eq:relative-scalars}; only the coefficient $Q$ is specialised to $q^M$.

\begin{proposition}[Specialised finite formulas]\label{prop:specialised-finite}
On $\cR_M(\{x\}\cup Y)$,
\begin{align}
\partial^{M,Y,\mathrm R}_{x,q}(y_IA)
 &=y_I\partial^{\mathrm{sc},M,Y}_{x,q}(A),\label{eq:super-finite-y}\\
\partial^{M,Y,\mathrm R}_{x,q}((x\wedge y_I)B)
 &=(-1)^{|I|}\qnum{M-|I|}y_IB
   +q^{M-|I|}(x\wedge y_I)\partial^{\mathrm{sc},M,Y}_{x,q}(B),\label{eq:super-finite-x}
\end{align}
where
\begin{equation}\label{eq:super-scalar-operator}
\partial^{\mathrm{sc},M,Y}_{x,q}
 =(1+q)x\delta_{q^2}^{(r)}
 +2T_{q^2}^{(r)}\sum_{i=1}^Ny_i\delta_q^{(s_i)}.
\end{equation}
The same formulas hold in every faithful coordinate block through \eqref{eq:coordinate-conjugation}.
\end{proposition}

\begin{proof}
Apply Theorem~\ref{thm:base-change} to \eqref{eq:finite-y}--\eqref{eq:finite-x}.  The coordinate statement follows from Proposition~\ref{prop:faithful-super-block} and Definition~\ref{def:coordinate-super-derivative}.
\end{proof}

\begin{example}[One auxiliary vector]\label{ex:one-aux}
Let $Y=\{y\}$ and write
\[
        r=x^2,\qquad s=\{x,y\},\qquad t=y^2,
\]
so that $c_{11}=\{y,y\}=2t$.  Then
\[
\partial^{\mathrm{sc},M,\{y\}}_{x,q}
 =(1+q)x\delta_{q^2}^{(r)}+2T_{q^2}^{(r)}y\delta_q^{(s)},
\]
and
\[
        \partial^{M,\{y\},\mathrm R}_{x,q}(x\wedge y)=-\qnum{M-1}y.
\]
\end{example}

\begin{example}[Right-action convention check]\label{ex:right-low-degree}
For $Y=\{y\}$, let $w=x\wedge y$ and $s=\{x,y\}$. The specialised right-action rules give
\[
\begin{array}{c|ccccccc}
F&1&x&y&x^2&s&w&yx\\ \hline
[F]\partial^{M,\mathrm R}_{x,q}&0&[M]_q&0&[2]_qx&2y&-[M-1]_qy&([M-1]_q+1)y
\end{array}
\]
and, since $xy=w+s/2$,
\[
 [xy]\partial^{M,\mathrm R}_{x,q}=(1-[M-1]_q)y.
\]
At $q\to1$ these recover $[yx]\partial_x=My$ and $[xy]\partial_x=(2-M)y$. In particular, $[x\wedge y]\partial_x=-(M-1)y$.
\end{example}

\section{Localised decompositions}\label{sec:decompositions}

\subsection{Exterior Green decomposition}

For finite $Y$, put
\[
 \cA_x^M(Y):=\R(q)[r,s_1,\ldots,s_N,c_{ij}:1\le i\le j\le N]
\]
and set
\[
 \cY_{M;x;Y}:=\bigoplus_I y_I\cA_x^M(Y),
 \qquad
 \cX_{M;x;Y}:=\bigoplus_I(x\wedge y_I)\cA_x^M(Y).
\]
Exterior creation is
\begin{equation}\label{eq:creation}
 \cC_x^Y(y_IA)=(x\wedge y_I)A,
 \qquad
 \cC_x^Y((x\wedge y_I)B)=0.
\end{equation}
After specialisation, define
\begin{equation}\label{eq:H-super}
 \cH_{M;x,q}^Y
 :=\partial^{M,Y,\mathrm R}_{x,q}\cC_x^Y
   +\cC_x^Y\partial^{M,Y,\mathrm R}_{x,q}.
\end{equation}
The scalar $q$-Euler operator is
\begin{equation}\label{eq:q-euler-super}
 E^{\mathrm{sc},M,Y}_{x,q}
 =(1+q)M_r\delta_{q^2}^{(r)}
 +T_{q^2}^{(r)}\sum_{i=1}^NM_{s_i}\delta_q^{(s_i)},
\end{equation}
where $M_u$ denotes multiplication by $u$.  For $p\ge0$, put
\begin{equation}\label{eq:Delta-super}
 \Delta^Y_{M,p,q}
 :=(-1)^p\bigl(\qnum{M-p}I+q^{M-p}E^{\mathrm{sc},M,Y}_{x,q}\bigr).
\end{equation}
On $r^as^\beta P(c)$ its eigenvalue is
\begin{equation}\label{eq:mu-super}
 \mu^M_{p,a,\beta}(q)
 =(-1)^p\left(\qnum{M-p}+q^{M-p}
 \left(\qnum{2a}+q^{2a}\sum_{i=1}^N\qnum{\beta_i}\right)\right).
\end{equation}
Before specialisation, let $\Omega^Y_x$ be the central multiplicative set generated by the formal factors obtained from \eqref{eq:mu-super} by replacing $M$ with $d$.  We call the block Green-nonresonant at $M$ when
\[
        0\notin\sigma_M(\Omega^Y_x).
\]
In that case put $\Omega^Y_{M;x}:=\sigma_M(\Omega^Y_x)$.

\begin{theorem}[Triangular exterior anticommutator]\label{thm:triangular-super}
Relative to $\cY_{M;x;Y}\oplus\cX_{M;x;Y}$, the operator \eqref{eq:H-super} has the form
\begin{equation}\label{eq:H-super-block}
 \cH_{M;x,q}^Y
 =\begin{pmatrix}
   \displaystyle\bigoplus_I\Delta^Y_{M,|I|,q}&0\\
   \mathcal N^Y_{M;x,q}&\displaystyle\bigoplus_I\Delta^Y_{M,|I|,q}
  \end{pmatrix}
\end{equation}
for a certain lower off-diagonal operator $\mathcal N^Y_{M;x,q}$.
\end{theorem}

\begin{proof}
Let $\operatorname{pr}_{0,x}$ denote projection onto the sector without an exterior $x$.  To see the diagonal block, project $(x\wedge y_I)\partialsc(A)$.  In the unspecialised algebra,
\[
 \operatorname{pr}_{0,x}\bigl((x\wedge y_I)\partialsc(A)\bigr)
 =(-1)^{|I|}y_IE^{\mathrm{sc},Y}_{x,q}A.
\]
Indeed, writing $\partialsc(A)=x\alpha+\sum_jy_j\beta_j$, the relevant contractions are
\[
 \operatorname{pr}_{0,x}((x\wedge y_I)x)=(-1)^{|I|}ry_I,
 \qquad
 \operatorname{pr}_{0,x}((x\wedge y_I)y_j)=(-1)^{|I|}\frac{s_j}{2}y_I.
\]
Equations \eqref{eq:finite-y}--\eqref{eq:finite-x} give the formal triangular matrix whose diagonal is obtained from \eqref{eq:Delta-super} by replacing $M$ with $d$.  Applying $Q\mapsto q^M$ yields \eqref{eq:H-super-block}.
\end{proof}

Define the graded inverse by
\begin{equation}\label{eq:graded-green-super}
 G^{\mathrm{gr},Y}_{M;x,q}
 \bigl(y_Ir^as^\beta P(c)\bigr)
 =(\mu^M_{|I|,a,\beta}(q))^{-1}y_Ir^as^\beta P(c),
\end{equation}
with the same formula on the $x$-exterior sector.

\begin{theorem}[Localised exterior Green operator]\label{thm:green-super}
Assume the finite block is Green-nonresonant at $M$.  On $\cR_M(\{x\}\cup Y)_{\Omega^Y_{M;x}}$, the inverse of \eqref{eq:H-super} is
\begin{equation}\label{eq:green-super-matrix}
 \cG^Y_{M;x,q}
 =\begin{pmatrix}
   G^{\mathrm{gr},Y}_{M;x,q}&0\\
   -G^{\mathrm{gr},Y}_{M;x,q}\mathcal N^Y_{M;x,q}G^{\mathrm{gr},Y}_{M;x,q}
   &G^{\mathrm{gr},Y}_{M;x,q}
  \end{pmatrix}.
\end{equation}
If $Y\subset Z$, the corresponding Green inverses agree on the smaller finite radial algebra after localisation by both sets of central factors.
\end{theorem}

\begin{proof}
Formula \eqref{eq:green-super-matrix} is the inverse of the lower triangular matrix \eqref{eq:H-super-block}.  If $Y\subset Z$, the smaller algebra is invariant under both anticommutators, their restrictions agree by Theorem~\ref{thm:direct-limit}, and the graded inverses act by the same factors on monomials supported in $Y$.  Hence the inverses agree on the smaller block.
\end{proof}

\begin{corollary}[Exterior Green decomposition]\label{cor:green-decomposition}
Define
\begin{equation}\label{eq:green-projectors-super}
 P_{\partial,M}^Y
 :=\partial^{M,Y,\mathrm R}_{x,q}\cC_x^Y\cG^Y_{M;x,q},
 \qquad
 P_{\cC,M}^Y
 :=\cC_x^Y\partial^{M,Y,\mathrm R}_{x,q}\cG^Y_{M;x,q}.
\end{equation}
Then these are complementary projections and
\begin{equation}\label{eq:green-decomposition-super}
 \cR_M(\{x\}\cup Y)_{\Omega^Y_{M;x}}
 =\im P_{\partial,M}^Y\oplus\im P_{\cC,M}^Y,
\end{equation}
with
\[
 \im P_{\partial,M}^Y\subseteq\im\partial^{M,Y,\mathrm R}_{x,q},
 \qquad
 \im P_{\cC,M}^Y\subseteq\im\cC_x^Y.
\]
The same statement holds in every faithful coordinate block.
The finite projectors are compatible whenever both blocks are
Green-nonresonant. The corresponding all-support nonresonance hypothesis
cannot hold at integral superdimension; see Remark~\ref{rem:global-green-vacuous}.
\end{corollary}

\begin{proof}
The identity $I=\cH_{M;x,q}^Y\cG^Y_{M;x,q}$ expands as
\[
 F=\partial^{M,Y,\mathrm R}_{x,q}(\cC_x^Y\cG^Y_{M;x,q}F)
   +\cC_x^Y(\partial^{M,Y,\mathrm R}_{x,q}\cG^Y_{M;x,q}F).
\]
Write
\[
 \partial^{M,Y,\mathrm R}_{x,q}
 =\begin{pmatrix}A&\Delta\\C&D\end{pmatrix},
 \qquad
 \cC_x^Y=\begin{pmatrix}0&0\\I&0\end{pmatrix}
\]
relative to $\cY_{M;x;Y}\oplus\cX_{M;x;Y}$.  Substituting \eqref{eq:green-super-matrix} gives
\[
 P_{\partial,M}^Y=\begin{pmatrix}I&0\\D\Delta^{-1}&0\end{pmatrix},
 \qquad
 P_{\cC,M}^Y=\begin{pmatrix}0&0\\-D\Delta^{-1}&I\end{pmatrix}.
\]
They are idempotent, have zero product in both orders, and sum to the identity.  The coordinate statement follows by conjugation through the faithful representation.  The finite-support compatibility is the assertion of
Theorem~\ref{thm:green-super}.
\end{proof}

\begin{remark}[Why a global Green-nonresonance hypothesis fails]\label{rem:global-green-vacuous}
If the ambient label set has arbitrarily large finite subsets, for every
integer $M$ some finite block contains a zero diagonal Green factor.
Indeed, for $M>0$ take a blade of rank $p=M$, with $a=0$ and
$\beta=0$; for $M=0$ take $p=a=0$ and $\beta=0$.
For $M=-2\ell<0$ take $p=0,a=\ell,\beta=0$, and for
$M=-(2\ell+1)<0$ take $p=0,a=\ell$ and $\beta_1=1$ in a block
with one auxiliary label.  In each case, equation~\eqref{eq:mu-super}
vanishes identically.  Thus the finite Green theorem is meaningful, but
its nonresonance hypotheses cannot hold simultaneously on all degrees
and finite supports at fixed integral $M$.
\end{remark}

\subsection{Determinant-localised right-monogenic Fischer decomposition}

Let $R_xF=Fx$. The use of the right multiplication is required for consistency with the classical right Fischer theory.  Give the finite radial algebra the $x$-degree
\[
 \deg_xr=2,
 \quad\deg_xs_i=1,
 \quad\deg_xc_{ij}=0,
 \quad\deg_xy_I=0,
 \quad\deg_x(x\wedge y_I)=1.
\]
Let $H^Y_{M,k}$ be the degree-$k$ component after specialisation and put
\[
 \mathbb T^M_Y:=\R(q)[c_{ij}:1\le i\le j\le N].
\]
It is a finite free $\mathbb T^M_Y$-module.  Define
\begin{equation}\label{eq:K-D-super}
 \cK^Y_{M,k}
 :=\partial^{M,Y,\mathrm R}_{x,q}R_x\bigm|_{H^Y_{M,k}},
 \qquad
 D^Y_{M,k}:=\det\cK^Y_{M,k}\in\mathbb T^M_Y.
\end{equation}

\begin{lemma}[Injectivity of right multiplication]\label{lem:Lx-injective-super}
The map $R_x$ is injective on every finite formal radial superspace block and remains injective after central localisation or evaluation at $q_0\in(0,1)$.
\end{lemma}

\begin{proof}
If $Fx=0$, then $Fx^2=Fr=0$.  The relative exterior decomposition makes the radial algebra a free module over its scalar polynomial algebra, and the scalar variable $r$ is not a zero divisor.  Hence $F=0$.  The same free-module argument applies after specialisation, numerical evaluation, and central localisation, because the evaluated scalar coefficient ring is again a polynomial ring and $r$ remains an indeterminate.
\end{proof}

\begin{theorem}[Finite determinant-localised Fischer decomposition]\label{thm:finite-fischer-super}
If $D^Y_{M,k}\ne0$, then after localisation by this determinant,
\begin{equation}\label{eq:finite-fischer-super}
\begin{aligned}
 (H^Y_{M,k+1})_{D^Y_{M,k}}
 &=\Ker\!\left(\partial^{M,Y,\mathrm R}_{x,q}:
      (H^Y_{M,k+1})_{D^Y_{M,k}}
      \longrightarrow(H^Y_{M,k})_{D^Y_{M,k}}\right)\\
 &\quad\oplus R_x(H^Y_{M,k})_{D^Y_{M,k}}.
\end{aligned}
\end{equation}
The projection onto the right-monogenic summand is
\begin{equation}\label{eq:monogenic-projection-super}
 \Pi^Y_{M,k+1}
 =I-R_x(\cK^Y_{M,k})^{-1}\partial^{M,Y,\mathrm R}_{x,q}.
\end{equation}
The same decomposition holds in a faithful coordinate block whenever the specialised determinant is nonzero.
\end{theorem}

\begin{proof}
For $F\in H^Y_{M,k+1}$ set
\[
 U=(\cK^Y_{M,k})^{-1}\partial^{M,Y,\mathrm R}_{x,q}F,
 \qquad
 F_0=F-R_xU.
\]
Then $\partial^{M,Y,\mathrm R}_{x,q}F_0=0$.  If $R_xU$ also belongs to the kernel, then $\cK^Y_{M,k}U=0$, so $U=0$.  This proves existence, directness, and \eqref{eq:monogenic-projection-super}.  The coordinate assertion follows by conjugation.
\end{proof}

\begin{theorem}[Compatibility of monogenic projections]\label{thm:projection-compatibility}
Let $Y\subset Z$ and suppose that $D^Y_{M,k}$ and $D^Z_{M,k}$ are nonzero.  After localising by both determinants,
\[
 (\cK^Z_{M,k})^{-1}\bigm|_{H^Y_{M,k}}
 =(\cK^Y_{M,k})^{-1},
 \qquad
 \Pi^Z_{M,k+1}\bigm|_{H^Y_{M,k+1}}
 =\Pi^Y_{M,k+1}.
\]
Consequently, if all finite determinants are nonzero at a fixed $M$, the finite decompositions induce the direct-limit decomposition
\begin{equation}\label{eq:global-fischer-super}
 \cR_M(S)_{\Omega^{\mathrm{mon}}_{M;x}}
 =\Ker\partial^{M,\mathrm R}_{x,q}
   \oplus R_x\cR_M(S)_{\Omega^{\mathrm{mon}}_{M;x}},
\end{equation}
where $\Omega^{\mathrm{mon}}_{M;x}$ is generated by the specialised finite determinants.
\end{theorem}

\begin{proof}
The submodule $H^Y_{M,k}$ is invariant under $\cK^Z_{M,k}$, and the restriction is $\cK^Y_{M,k}$ by finite-support compatibility.  Both restrictions are invertible after localisation; uniqueness of the inverse gives the first equality.  Formula \eqref{eq:monogenic-projection-super} then gives the second.  Every element has finite support and finite $x$-degree, so the compatible finite projections define the direct limit.
\end{proof}

\begin{remark}
At a fixed coordinate dimension, Proposition~\ref{prop:faithful-super-block} gives actual coordinate versions for finite supports satisfying $|Y|+1\le m$.  Equation \eqref{eq:global-fischer-super} concerns the universal formal radial algebra. Its coordinate interpretation requires faithful blocks or a separately established kernel-invariance property.
\end{remark}

\section{Integral-superdimension resonance}\label{sec:resonance}

The regular form \eqref{eq:regular-formal-super-algebra} has coefficients in $\mathbb A_q$.  For $q_0\in(0,1)$, let
\[
 \operatorname{ev}_{q_0}:\mathbb A_q\longrightarrow\R,
 \qquad q\longmapsto q_0.
\]
Define
\begin{equation}\label{eq:evaluated-radial-algebra}
 \cR_{M,q_0}(S)
 :=\R\otimes_{\operatorname{ev}_{q_0},\mathbb A_q}
   \cR_M^{\mathbb A}(S).
\end{equation}
The regular specialised derivative preserves $\cR_M^{\mathbb A}(S)$, so base change defines $\partial^{M,Y,\mathrm R}_{x,q_0}$ on \eqref{eq:evaluated-radial-algebra}.  We write $H^Y_{M,k}(q_0)$ for the evaluated homogeneous module.  For a finite auxiliary set $Y$, put
\[
 \mathbb T_Y(q_0):=\R[c_{ij}:1\le i\le j\le |Y|].
\]

\subsection{Complete one-vector kernel and cokernel}

For $Y=\varnothing$, the evaluated one-vector formulas are
\begin{equation}\label{eq:one-vector-evaluated}
 \partial^{M,\mathrm R}_{x,q_0}(x^{2a})
 =\qnumat{2a}{q_0}x^{2a-1},
 \qquad
 \partial^{M,\mathrm R}_{x,q_0}(x^{2a+1})
 =\qnumat{M+2a}{q_0}x^{2a}.
\end{equation}
For an integer $r$ and $0<q_0<1$, one has $\qnumat r{q_0}=0$ if and only if $r=0$.

\begin{theorem}[Exact one-vector defect]\label{thm:one-vector-exact}
Let $M\in\mathbb Z$ and $0<q_0<1$.

\begin{enumerate}[label=\textup{(\roman*)}]
\item If $M\notin-2\N_0$, then
\[
 \Ker\partial^{M,\mathrm R}_{x,q_0}=\R\,1,
 \qquad
 \im\partial^{M,\mathrm R}_{x,q_0}=\cR_{M,q_0}(\{x\}).
\]
Equivalently, there is an exact sequence
\[
 0\longrightarrow\R\,1
 \longrightarrow\cR_{M,q_0}(\{x\})
 \xrightarrow{\,\partial^{M,\mathrm R}_{x,q_0}\,}
 \cR_{M,q_0}(\{x\})
 \longrightarrow0.
\]

\item If $M=-2\ell$ with $\ell\in\N_0$, then
\[
 \Ker\partial^{M,\mathrm R}_{x,q_0}
 =\R\,1\oplus\R\,x^{2\ell+1},
\]
\[
 \im\partial^{M,\mathrm R}_{x,q_0}
 =\operatorname{span}_{\R}
   \{x^j:j\ge0,\ j\ne2\ell\},
 \qquad
 \operatorname{coker}\partial^{M,\mathrm R}_{x,q_0}
 \cong\R\,x^{2\ell}.
\]
Thus
\[
 0\longrightarrow
 \R\,1\oplus\R\,x^{2\ell+1}
 \longrightarrow\cR_{M,q_0}(\{x\})
 \xrightarrow{\,\partial^{M,\mathrm R}_{x,q_0}\,}
 \cR_{M,q_0}(\{x\})
 \longrightarrow\R\,x^{2\ell}
 \longrightarrow0
\]
is exact.
\end{enumerate}
In both cases the algebraic index
$\dim\Ker\partial^{M,\mathrm R}_{x,q_0}
-\dim\operatorname{coker}\partial^{M,\mathrm R}_{x,q_0}$
is equal to $1$.
\end{theorem}

\begin{proof}
The monomials $1,x,x^2,\ldots$ form a basis of the one-vector radial algebra.  By \eqref{eq:one-vector-evaluated}, every even positive power maps to the preceding odd power with nonzero coefficient.  The odd power $x^{2a+1}$ maps to $x^{2a}$ with a zero coefficient precisely when $M+2a=0$.  Such an $a\in\N_0$ exists precisely for $M=-2\ell$, and then it is unique, namely $a=\ell$.  The kernel, image, and cokernel statements follow degree by degree.  The index calculation is immediate.
\end{proof}

\begin{corollary}[One-vector Fischer alternative]\label{cor:one-vector-fischer-alternative}
Under the assumptions of Theorem~\ref{thm:one-vector-exact},
\[
 \cR_{M,q_0}(\{x\})
 =\Ker\partial^{M,\mathrm R}_{x,q_0}+x\cR_{M,q_0}(\{x\}).
\]
The sum is direct if $M\notin-2\N_0$.  If $M=-2\ell$, then
\[
 \Ker\partial^{M,\mathrm R}_{x,q_0}
 \cap x\cR_{M,q_0}(\{x\})
 =\R\,x^{2\ell+1}.
\]
If $M=m-2n$ and $m\ge1$, the same statements hold in the represented one-vector coordinate block, and $x^{2\ell+1}$ is a nonzero radial superpolynomial.
\end{corollary}

\begin{proof}
The ideal $x\cR_{M,q_0}(\{x\})$ is spanned by all positive powers of $x$.  The assertions follow from the kernel description in Theorem~\ref{thm:one-vector-exact}.  The coordinate statement follows from Proposition~\ref{prop:faithful-super-block} with $|T|=1$.
\end{proof}

The homogeneous one-vector determinant is therefore
\begin{equation}\label{eq:one-vector-determinant}
 D^{\varnothing}_{M,k}
 =\begin{cases}
   \qnum{M+k},&k\text{ even},\\
   \qnum{k+1},&k\text{ odd}.
  \end{cases}
\end{equation}
Theorem~\ref{thm:one-vector-exact} identifies the full defect behind the vanishing of the even factor, rather than only detecting singularity of one homogeneous block.

\subsection{Exterior-blade Green resonance}

\begin{proposition}[An exterior Green obstruction]\label{prop:green-resonance}
Let $M\in\N_0$. If a finite auxiliary block contains a blade $y_I$ of rank $|I|=M$, then
\[ \mu^M_{M,0,0}(q)=0,\qquad
 \cH^Y_{M;x,q}(x\wedge y_I)=0.\]
Consequently the exterior anticommutator is not invertible on this block without localisation.
\end{proposition}
\begin{proof}
The diagonal factor is $(-1)^M[M-M]_q=0$. By the pure-blade derivative,
$\partial^{M,Y,\mathrm R}_{x,q}(x\wedge y_I)=(-1)^M[0]_qy_I=0$, and exterior creation annihilates $x\wedge y_I$.
\end{proof}

\subsection{Degree-zero support spectrum}

Fix $Y=\{y_1,\ldots,y_N\}$. For a relative PBW monomial $y_Ir^as^\beta$ or $(x\wedge y_I)r^as^\beta$, put
\[
 \operatorname{supp}_x(I,\beta):=I\cup\{i:\beta_i>0\}.
\]
Let $\mathcal F^J_{M,k}(Y)$ denote the span over $\mathbb T_Y^M$ of the degree-$k$ monomials with support contained in $J$.

\begin{proposition}[Support filtration]\label{prop:support-filtration-super}
The right Fischer operator $\cK^Y_{M,k}=\partial^{M,Y,\mathrm R}_{x,q}R_x$ preserves every $\mathcal F^J_{M,k}(Y)$. Consequently its determinant factors into the determinants of the induced exact-support operators. These induced maps are obtained by coefficient extension from the smaller auxiliary set $\{y_j:j\in J\}$.
\end{proposition}
\begin{proof}
Right multiplication by $x$ either preserves an exterior blade, or contracts one of its auxiliary factors, introducing the corresponding mixed scalar $s_i$. Neither action enlarges the union defining the support. The scalar Jackson differences change only $r$ or remove a factor of $s_i$ while introducing $y_i$. The blade parity sign does not affect the support. Thus the ordered exact-support basis gives a block-triangular matrix and the asserted determinant factorisation, as in \cite{BarseghyanBorySchneiderZhang2026}.
\end{proof}

\begin{lemma}[Alternating blade identity]\label{lem:alternating-blade-identity}
For $J=\{i_1<\cdots<i_p\}$, one has the identity in the ordinary radial algebra
\begin{equation}\label{eq:alternating-blade-identity}
 \sum_{\nu=1}^p(-1)^{\nu-1}y_{J\setminus\{i_\nu\}}y_{i_\nu}
 =(-1)^{p-1}p\,y_J.
\end{equation}
\end{lemma}
\begin{proof}
The exterior parts of the summands all have sign $(-1)^{p-1}$. Every scalar contraction corresponding to a fixed pair occurs twice with opposite signs; they cancel because the contraction is symmetric. The exterior terms sum to $(-1)^{p-1}p\,y_J$.
\end{proof}

\begin{theorem}[Exact right degree-zero blade spectrum]\label{thm:exact-blade-spectrum}
For every $J\subseteq\{1,\ldots,N\}$, $p=|J|$, the degree-zero Fischer operator is \emph{exactly diagonal}:
\begin{equation}\label{eq:exact-blade-action}
 \cK^Y_{M,0}(y_J)=F_{M,p}(q)y_J,\qquad
 F_{M,p}(q)=\begin{cases}[M]_q,&p=0,\\{}[M-p]_q+p,&p\ge1.\end{cases}
\end{equation}
In particular,
\begin{equation}\label{eq:degree-zero-determinant}
 D^Y_{M,0}=\prod_{p=0}^N F_{M,p}(q)^{\binom Np}.
\end{equation}
\end{theorem}
\begin{proof}
For $p\ge1$, right multiplication in the central-coefficient PBW convention gives
\[
 R_xy_J=y_Jx=(-1)^p\left(x\wedge y_J-\sum_{\nu=1}^p(-1)^{\nu-1}\frac{s_{i_\nu}}2 y_{J\setminus\{i_\nu\}}\right).
\]
The first term contributes $(-1)^p\cdot(-1)^p[M-p]_qy_J=[M-p]_qy_J$. Since $\partial^{\mathrm{sc},M,Y}_{x,q}(s_i/2)=y_i$, the sum contributes
\[
 (-1)^{p+1}\sum_{\nu=1}^p(-1)^{\nu-1}y_{J\setminus\{i_\nu\}}y_{i_\nu}
 =p\,y_J
\]
by Lemma~\ref{lem:alternating-blade-identity}. All scalar contraction terms cancel \emph{before} passing to any support quotient. For $p=0$, use $\partial^{M,\mathrm R}_{x,q}x=[M]_q$. Multiplying the diagonal entries yields the determinant.
\end{proof}

\begin{theorem}[Complete numerical degree-zero resonance]\label{thm:support-root-classification}
Fix $M\in\mathbb Z$, $0<q<1$, and $p\ge1$.
If $M>0$, the factor $F_{M,p}$ is strictly positive when $p\le M$; when $p>M$, it has a unique, simple zero $q_{M,p}\in(0,1)$.
If $M\le0$, no factor $F_{M,p}$ with $p\ge1$ vanishes in $(0,1)$. The empty-support factor $[M]_q$ vanishes identically precisely for $M=0$.
For each fixed $M>0$, the numbers $q_{M,p}$, $p>M$, are pairwise distinct, and increase strictly with $p$.
\end{theorem}
\begin{proof}
For $p\le M$, $[M-p]_q\ge0$, giving $F_{M,p}\ge p>0$. If $p>M$, write $s=p-M>0$; then
\begin{equation}\label{eq:resonance-factor}
 F_{M,p}(q)=p-\sum_{j=1}^{s}q^{-j}
 =M-\sum_{j=1}^{s}(q^{-j}-1).
\end{equation}
The sum on the right of the second expression is strictly decreasing in $q\in(0,1)$ from $+\infty$ to $0$. For $M>0$ the equation has exactly one simple root, while for $M\le0$ it has none. At fixed $M>0$ and $q\in(0,1)$, increasing $s$ strictly increases this sum. Hence at most one rank $p$ can be resonant for a given $q$, and maintaining the value $M$ at a larger $s$ requires a strictly larger $q$.
For $p=0$, $[M]_q=0$ at $0<q<1$ exactly when $M=0$.
\end{proof}

\begin{corollary}[Exact singular multiplicity and Fischer failure]\label{cor:support-fischer-failure}
Let $M>0$, $p>M$, $N\ge p$, and $q_0=q_{M,p}$. Then, over the polynomial coefficient ring $\mathbb T_Y(q_0)$,
\[
 \Ker\cK^Y_{M,0}(q_0)=\bigoplus_{|J|=p}\mathbb T_Y(q_0)y_J,
 \qquad \operatorname{rank}_{\mathbb T_Y(q_0)}\Ker\cK^Y_{M,0}(q_0)=\binom Np.
\]
The $\binom Np$ nonzero vectors $R_xy_J=y_Jx$ lie in the kernel of the right derivative; the unlocalised right Fischer direct sum is therefore not direct. If the bosonic dimension satisfies $m\ge N+1$ and $M=m-2n$, these identities hold on the faithful coordinate superspace block as well.
\end{corollary}
\begin{proof}
Theorem~\ref{thm:exact-blade-spectrum} is diagonal, and Theorem~\ref{thm:support-root-classification} says that its only vanishing factor at $q_0$ has rank $p$; all other diagonal factors are nonzero real constants. By injectivity of $R_x$, the images $R_xy_J$ are nonzero and linearly independent. The finite coordinate realisation is faithful by Proposition~\ref{prop:faithful-super-block} when $m\ge N+1$.
\end{proof}

\begin{example}[First resonances]\label{ex:first-support-roots}
For $M=1$ and $p=2$ one obtains
\[
 F_{1,2}(q)=[-1]_q+2=2-q^{-1},
 \qquad q_{1,2}=\frac12.
\]
For $M=1$ and $p=3$ the equation is $q^{-1}+q^{-2}=3$, giving $q_{1,3}=(1+\sqrt{13})/6$. These roots are distinct, and each produces an explicit family of singular right-monogenic elements.
\end{example}

\subsection{The two-vector benchmark}

The two-vector determinant follows by specialisation of the calculation for $\partial^{Y,\mathrm R}_{x,q}R_x$ in \cite[Section~6]{BarseghyanBorySchneiderZhang2026}. Its factors are as follows. Put
\begin{align*}
 A^M_{k,a}&=[M]_q+q^M([2a]_q+2q^{2a}[k-2a]_q),\\
 B^{M,\mathrm R}_{k,a}&=[M-1]_q+q^{M-1}([k]_q+q^{2a+2}[k-1-2a]_q),\\
 C^{M,\mathrm R}_k&=[M-1]_q+[k+1]_q+q^{M-1}[k]_q.
\end{align*}

\begin{proposition}[Two-vector determinant]\label{prop:two-vector-determinant}
With empty products interpreted as $1$, one has
\begingroup\small
\begin{equation}\label{eq:two-vector-determinant}
\begin{aligned}
D^{\{y\}}_{M,k}
&=([M]_q+q^M[k]_q) C^{M,\mathrm R}_k
  \prod_{a=0}^{k/2-1}[2a+2]_q A^M_{k,a}
  \prod_{a=0}^{k/2-1}[2a+2]_q B^{M,\mathrm R}_{k,a}, &&k\text{ even},\\
D^{\{y\}}_{M,k}
&=[k+1]_q C^{M,\mathrm R}_k
  \prod_{a=0}^{(k-1)/2}[2a+2]_q A^M_{k,a}
  \prod_{a=0}^{(k-3)/2}[2a+2]_q B^{M,\mathrm R}_{k,a}, &&k\text{ odd}.
\end{aligned}
\end{equation}
\endgroup
If $M\ge1$ and $0<q<1$, every factor is positive and the two-vector Fischer decomposition is available without localisation in every degree $k$.
\end{proposition}
\begin{proof}
This is the right-multiplication matrix calculation reported in \cite[Theorem~6.1]{BarseghyanBorySchneiderZhang2026} after $Q=q^M$.  For $M\ge1$, all $q$-numbers appearing in $A^M_{k,a}$, $B^{M,\mathrm R}_{k,a}$ and $C^{M,\mathrm R}_k$ are nonnegative and each relevant factor contains a strictly positive summand. The remaining $[2a+2]_q$ and $[k+1]_q$ factors are positive. Thus $D^{\{y\}}_{M,k}>0$.
\end{proof}

\section{Conclusion}

The right $q$-radial derivative admits a finite-support calculus at
integral superdimension $M=m-2n$. Formal specialisation is available
for every integer $M$, and faithful coordinate realisations are obtained
on finite blocks satisfying the stated dimension condition. The exterior
Green and right Fischer decompositions hold under explicit
nonvanishing hypotheses.

For one vector, the kernel and cokernel identify the exceptional
behaviour at nonpositive even superdimension. With arbitrary finite
auxiliary support, the degree-zero Fischer operator is diagonal on
exterior blades, with nonempty-support eigenvalues $[M-p]_q+p$.
For positive $M$, each rank $p>M$ has a unique numerical resonance;
the resonance values are simple and increase strictly with $p$.
The kernel multiplicities and explicit singular right-monogenic vectors
exhibit the failure of the direct-sum Fischer decomposition at these
parameters. The two-vector determinant remains positive for
positive integral $M$ and $0<q<1$.

\section*{Funding}

This work was supported by the University of Ostrava under Grant
No.~SGS05/P\v{R}F/2026. Yifan Zhang is also co-funded by the European Union
under the REFRESH -- Research Excellence For REgion Sustainability and
High-tech Industries project, No.~CZ.10.03.01/00/22003/0000048, via the
Operational Programme Just Transition.

\bibliographystyle{amsplain}
\bibliography{references}

@misc{BarseghyanBorySchneiderZhang2026,
  author = {Barseghyan (Schneiderov{\'a}), Diana and Bory-Reyes, Juan and Schneider, Baruch and Zhang, Yifan},
  title = {Intrinsic {$q$}-{Radial} Vector Derivatives and Localized {Fischer} Decompositions on {Radial} Algebras},
  year = {2026},
  eprint = {2605.00775},
  archivePrefix = {arXiv},
  note = {\href{https://doi.org/10.48550/arXiv.2605.00775}{doi:10.48550/arXiv.2605.00775}}
}

@article{DeBieSommen2007AnnPhys,
  author = {De Bie, H. and Sommen, F.},
  title = {A {Clifford} analysis approach to superspace},
  journal = {Annals of Physics},
  volume = {322},
  number = {12},
  pages = {2978--2993},
  year = {2007},
  note = {\href{https://doi.org/10.1016/j.aop.2007.04.012}{doi:10.1016/j.aop.2007.04.012}}
}

@article{DeBieSommen2007CorrectRules,
  author = {De Bie, H. and Sommen, F.},
  title = {Correct rules for {Clifford} calculus on superspace},
  journal = {Advances in Applied Clifford Algebras},
  volume = {17},
  pages = {357--382},
  year = {2007},
  note = {\href{https://doi.org/10.1007/s00006-007-0042-y}{doi:10.1007/s00006-007-0042-y}}
}

@article{DeSchepperGuzmanSommen2017,
  author = {De Schepper, H. and Guzm{\'a}n Ad{\'a}n, A. and Sommen, F.},
  title = {The radial algebra as an abstract framework for orthogonal and {Hermitian} {Clifford} analysis},
  journal = {Complex Analysis and Operator Theory},
  volume = {11},
  number = {5},
  pages = {1139--1172},
  year = {2017},
  note = {\href{https://doi.org/10.1007/s11785-016-0621-9}{doi:10.1007/s11785-016-0621-9}}
}

@article{Sommen1997,
  author = {Sommen, F.},
  title = {An algebra of abstract vector variables},
  journal = {Portugaliae Mathematica},
  volume = {54},
  number = {3},
  pages = {287--310},
  year = {1997}
}

\end{document}